\documentclass[reqno,11pt]{amsart}
\usepackage{amsfonts}
\usepackage{amsmath}
\usepackage{amsthm, amscd}
\usepackage{mathtools}
\usepackage{amssymb}
\usepackage{mathrsfs}
\usepackage{graphicx}
\usepackage{bbm}

\usepackage[alphabetic]{amsrefs}
\usepackage{etex}
\usepackage{hyperref}
\hypersetup{
  colorlinks = false
}

\usepackage{xypic}

\usepackage{fancyhdr}
\usepackage{color} 
\usepackage{overpic}
\usepackage{tikz-cd}

\allowdisplaybreaks

\newcommand\blfootnote[1]{%
  \begingroup
  \renewcommand\thefootnote{}\footnote{#1}%
  \addtocounter{footnote}{-1}%
  \endgroup
}

\theoremstyle{plain}\newtheorem{thm}{Theorem}[section]
\theoremstyle{plain}\newtheorem{cor}[thm]{Corollary}
\theoremstyle{plain}\newtheorem{prop}[thm]{Proposition}
\theoremstyle{plain}\newtheorem{lem}[thm]{Lemma}
\theoremstyle{plain}
\theoremstyle{plain}
\theoremstyle{plain}
\theoremstyle{plain}
\theoremstyle{plain}

\makeatletter
\newtheorem*{rep@theorem}{\rep@title}
\newcommand{\newreptheorem}[2]{%
\newenvironment{rep#1}[1]{%
 \def\rep@title{#2 \ref{##1}}%
 \begin{rep@theorem}}%
 {\end{rep@theorem}}}
\makeatother
\newreptheorem{theorem}{Theorem}

\theoremstyle{definition}
\newtheorem{defn}[thm]{Definition}

\newtheorem{exmp}[thm]{Example}

\theoremstyle{remark}
\newtheorem{rem}[thm]{Remark}

\numberwithin{equation}{section}

\DeclareMathOperator{\rank}{rank}
\DeclareMathOperator{\II}{I}

\DeclareMathOperator{\SHI}{SHI}
\DeclareMathOperator{\AHI}{AHI}

\DeclareMathOperator{\muu}{\mu^{orb}}

\DeclareMathOperator{\HFK}{\widehat{HFK}}
\DeclareMathOperator{\HF}{\widehat{HF}}

\DeclareMathOperator{\SFH}{SFH}
\DeclareMathOperator{\Eig}{Eig}
\DeclareMathOperator{\KHI}{KHI}

\newcommand{\bC}{\mathbb{C}}

\newcommand{\bQ}{\mathbb{Q}}

\newcommand{\bZ}{\mathbb{Z}}

\newcommand{\Ga}{\Gamma}

\newcommand{\al}{\alpha}

\newcommand{\ga}{\gamma}

\newcommand{\ra}{\rightarrow}

\newcommand{\xra}{\xrightarrow}

\newcommand{\bpf}{\begin{proof}}
\newcommand{\epf}{\end{proof}}
\newcommand{\bthm}{\begin{thm}}
\newcommand{\ethm}{\end{thm}}
\newcommand{\bprop}{\begin{prop}}
\newcommand{\eprop}{\end{prop}}
\newcommand{\bcor}{\begin{cor}}
\newcommand{\ecor}{\end{cor}}
\newcommand{\blem}{\begin{lem}}
\newcommand{\elem}{\end{lem}}
\newcommand{\bdefn}{\begin{defn}}
\newcommand{\edefn}{\end{defn}}
\newcommand{\bexmp}{\begin{exmp}}
\newcommand{\eexmp}{\end{exmp}}
\newcommand{\brem}{\begin{rem}}
\newcommand{\erem}{\end{rem}}

\newcommand{\bdia}{\begin{displaymath}\xymatrix}
\newcommand{\edia}{\end{displaymath}}
\newcommand{\beq}{\begin{equation*}\begin{aligned}}
\newcommand{\eeq}{\end{aligned}\end{equation*}}

\newcommand{\intg}{\mathbb{Z}}

\author{Zhenkun Li}
\address{Department of Mathematics, Stanford University, California, 94305, USA}
\email{zhenkun@stanford.edu}
\author{Yi Xie}
\address{Beijing International Center for Mathematical Research, Peking University, Beijing 100871, China}
\email{yixie@pku.edu.cn}
\author{Boyu Zhang}
\address{Department of Mathematics, Princeton University, New Jersey 08544, USA}
\email{bz@math.princeton.edu}

\title{On Floer minimal knots in sutured manifolds}

\begin{document}

\begin{abstract}
Suppose $(M, \gamma)$ is a balanced sutured manifold and $K$ is a rationally null-homologous knot in $M$. It is known that the rank of the sutured Floer homology of $M\backslash N(K)$ is at least twice the rank of the sutured Floer homology of $M$. This paper studies the properties of $K$ when the equality is achieved for instanton homology. As an application, we show that if $L\subset S^3$ is a fixed link and $K$ is a knot in the complement of $L$, then the instanton link Floer homology of $L\cup K$ achieves the minimum rank if and only if $K$ is the unknot in $S^3\backslash L$.
\end{abstract}

\maketitle

\section{Introduction}

\blfootnote{
The second author was supported by 
National Key R\&D Program of China SQ2020YFA0712800 and NSFC 12071005.}
\makeatother

Suppose $Y$ is a closed oriented 3-manifold and $K$ is a rationally null-homologous knot in $Y$. The spectral sequence in \cite[Lemma 3.6]{OS:HFK} implies the following\footnote{In the original setup of \cite{OS:HFK}, the knot $K$ is assumed to be null-homologous in $\bZ$ instead of in $\bQ$, but the same spectral sequence still works for the latter case.}:
\begin{equation}\label{eq_HFK>HF}
\rank\HFK(Y,K; \bZ/2)\ge \rank \HF(Y; \bZ/2).
\end{equation} 
The knot $K$ is called \emph{Floer simple} or \emph{Floer minimal} if \eqref{eq_HFK>HF} attains equality (see \cite{hedden2011floer, rasmussen2007lens}). 
The properties of Floer minimal knots have been studied in  \cite{Ni-Wu:floer_simple,rasmussen2017floer,GL:simple_knots}. 

More generally, suppose $(M,\ga)$ is a balanced sutured 3-manifold and $K$ is a knot in the interior of $M$ such that $[K] = 0 \in H_1(M,\partial M; \bQ)$. Let $N(K)\subset M$ be an open tubular neighborhood of $K$. Let $\ga_K$ be the union of $\ga$ and a pair of oppositely oriented meridians
on $\partial N(K)$. Then the following inequality holds for sutured Heegaard Floer homology:
\begin{equation}\label{eq_SFHK>SFH}
\rank\SFH(M\backslash N(K), \ga_K; \bZ/2)\ge 2\cdot \rank \SFH(M,\ga; \bZ/2).
\end{equation} 

\begin{rem}
The proof of \eqref{eq_SFHK>SFH} is the same as \eqref{eq_HFK>HF}: a spectral sequence similar to the one in \cite[Lemma 3.6]{OS:HFK} implies that
$$\rank\SFH(M\backslash N(K), \ga_K; \bZ/2)\ge\rank\SFH(M\backslash B^3, \ga\cup\delta; \bZ/2),$$
where $B^3$ is a $3$--ball inside $M$ and $\delta$ is a  simple closed curve on $\partial B^3$. Inequality \eqref{eq_SFHK>SFH} then follows from \cite[Proposition 9.14]{Juhasz-holo-disk}.
 \end{rem} 
 
 To see that \eqref{eq_HFK>HF} follows from \eqref{eq_SFHK>SFH}, take $M = Y\backslash B^3$ and let $\ga$ be a simple closed curve on $\partial B^3$.
 Another interesting special case of \eqref{eq_SFHK>SFH} is when $M$ is a link complement.  Let $L$ be a link in $S^3$, let $N(L)$ be an open tubular neighborhood of $L$, let $M=S^3\backslash N(L)$, and let $\ga\subset \partial M$ be the union of a pair of oppositely oriented meridians for every component of $L$.  In this case, by \cite[Proposition 9.2]{Juhasz-holo-disk}, inequality \eqref{eq_SFHK>SFH} is equivalent to: 
 \begin{equation}\label{eq_HFK_L_K}
\rank \HFK(L\cup K; \bZ/2)\ge 2 \cdot \rank \HFK(L; \bZ/2),
\end{equation} 
where $K$ is an arbitrary knot in $S^3\backslash L$.
It is natural to ask about the properties of the pair $(L,K)$ when the equality of \eqref{eq_HFK_L_K} holds. This question has been studied by Ni \cite{Ni:rank_unlink} and Kim \cite{Kim}. Ni \cite{Ni:rank_unlink} proved that if $L$ is the unlink and the equality of \eqref{eq_HFK_L_K} holds, then $L\cup K$ is the unlink. (This property was then used in \cite[Proposition 1.4]{Ni:rank_unlink} to prove that the rank of $\HFK$ detects the unlink.) Kim \cite[Proposition 4]{Kim} proved that if the equality of \eqref{eq_HFK_L_K} holds and $L\cup K$ is non-split, then $L$ is non-split and the linking number of $K$ with every component of $L$ is zero.

It is conjectured by Kronheimer and Mrowka \cite[Conjecture 7.25]{KM:suture} that sutured \emph{Heegaard} Floer homology is isomorphic to sutured \emph{instanton} Floer homology.
This paper studies an analogue of \eqref{eq_SFHK>SFH} for instanton Floer homology.
From now on, all instanton Floer homology groups are defined with $\bC$--coefficients.
An inequality on instanton Floer homology analogous to \eqref{eq_SFHK>SFH} was proved by Li-Ye \cite[Proposition 3.14]{LY-HeegaardDiagram}:  let $(M,\ga)$, $K$, and $\ga_K$ be as above, we have 
\begin{equation}\label{eq_SHIK>SHI}
\dim_\bC\SHI(M\backslash N(K), \ga_K)\ge 2\cdot \dim_\bC \SHI(M,\ga).
\end{equation}
\brem
To see that \eqref{eq_SHIK>SHI} follows from \cite[Proposition 3.14]{LY-HeegaardDiagram},
 remove a small $3$--ball $B^3$ from $M$ centered on $K$ such that $B^3\cap K$ is an arc. Let $\delta$ be a simple closed curve on $\partial B^3$ separating the two points in $\partial B^3\cap K$. Consider the balanced sutured manifold $(M\backslash B^3,\ga\cup\delta)$, and let $T= K\backslash B^3$ be a vertical tangle in $M\backslash B^3$. Then $(M\backslash N(K),\ga_K)$ can be obtained from $(M\backslash B^3,\ga\cup\delta)$ by removing a neighborhood of $T$ and adding a meridian of $T$ to the suture. By \cite[Proposition 3.14]{LY-HeegaardDiagram} and \cite[Proposition 4.15]{L18-contact}, 
$$\dim_\bC\SHI(M\backslash N(K), \ga_K)\ge \dim_\bC \SHI(M\backslash B^3,\ga\cup \delta) = 2\cdot \dim_\bC \SHI(M,\ga).$$
\erem

In this paper, we will study the properties of $K$ when the equality of \eqref{eq_SHIK>SHI} is achieved.
To simplify notation, we introduce the following definition:

\begin{defn}
\label{def_instanton_minimal}
Suppose $(M,\ga)$ is a sutured manifold. Suppose $K$ is a knot in $M$ such that $[K] = 0 \in H_1(M,\partial M;\mathbb{Q})$. Let $N(K)\subset M$ be an open tubular neighborhood of $K$. Let $\ga_K$ be the union of $\ga$ and a pair of oppositely oriented meridians
on $\partial N(K)$.   We say that $K$ is \emph{instanton Floer minimal}, if 
$$
\dim_\bC\SHI(M\backslash N(K),\ga_K) = 2\cdot \dim_\bC \SHI(M,\ga).
$$
\end{defn}

Suppose $(M,\ga)$ is a connected balanced sutured manifold. 
Let $\{\gamma_i\}_{1\le i \le s}$ be the components of $\ga$, then the homotopy class of every $\gamma_i$ defines a conjugacy class $[\gamma_i]$ in $\pi_1(M)$. 
Define
\begin{equation}
\label{eqn_def_pi'}
	\pi_1'(M):=\pi_1(M)/\langle [\gamma_1],\cdots, [\gamma_s]  \rangle,
\end{equation}
where $\langle [\gamma_1],\cdots, [\gamma_s]  \rangle$ is the normal subgroup of $\pi_1(M)$ generated by the conjugacy classes $[\gamma_1]$, $\cdots$, $[\gamma_s]$. 

The first main result of this paper is the following theorem.
 \begin{thm}
 \label{thm_instanton_minimal_in_Y}
Let $(M,\ga)$ be a connected balanced sutured manifold with 
$$\SHI(M,\ga)\neq 0.$$
 Suppose $M$ is decomposed as 
 \begin{equation}
 \label{eq_irr_decompostion_M}
 	M = M_1\# \cdots \# M_k \# Y,
 \end{equation}
 where each $M_i$ is an irreducible sutured manifold (in particular, $\partial M_i$ is non-empty), and $Y$ is a closed manifold. Here, $Y$ is allowed to be $S^3$ or reducible.
 Let $K$ be a knot in the interior of $M$ that represents the zero  element in $\pi_1'(M)$.  If $K$ is instanton Floer minimal, then $K$ is contained in $Y$ after isotopy. 
 \end{thm}
 
 \begin{rem}
 If the knot $K$ represents the zero element in $\pi_1'(M)$, then $[K] = 0 \in H_1(M,\partial M;\bQ)$, and hence Definition \ref{def_instanton_minimal} is applicable to $K$.	
 \end{rem}

Moreover, when $Y=S^3$ in \eqref{eq_irr_decompostion_M}, we have the following result.

\begin{thm}
\label{thm_instanton_minimal_unknot}
Suppose $(M,\ga)$ is a balanced taut sutured manifold or a connected sum of balanced taut sutured manifolds.
Let
$K$ be a knot in the interior of $M$ that represents the zero element in $\pi_1'(M)$. Then $K$ is instanton Floer minimal if and only if $K$ is the unknot.
\end{thm}

Suppose $L$ is a link in $S^3$ and $K$ is a knot in $S^3\backslash L$. By \eqref{eq_SHIK>SHI},
\begin{equation}
\label{eqn_KHI_unknot_inequality}
\dim_\bC\KHI(L\cup K)\ge 2\cdot\dim_\bC \KHI(L).
\end{equation}
The following result characterizes all the pairs $(L,K)$ such that the equality of \eqref{eqn_KHI_unknot_inequality} is achieved.
\begin{thm}
\label{thm_KHI_unknot}
Let $L\subset S^3$ be a link and $K\subset S^3 \backslash L$ be a knot. Then 
\begin{equation}
\label{eqn_KHI_unknot}
	\dim_\bC\KHI(L\cup K) =  2\cdot \dim_\bC \KHI(L)
\end{equation}
if and only if $K$ is the unknot in $S^3\backslash L$.
\end{thm}

Theorem \ref{thm_KHI_unknot} has the following immediate corollary:

\begin{cor}
\label{cor_unlink_rank_detection}
	Suppose $L$ is an $n$--component link in $S^3$. Then $L$ is an unlink if and only if $\dim_\bC \KHI(L) = 2^{n-1}$.
\end{cor}

\begin{proof}
Suppose $K$ is a knot in $S^3$, then by \cite[Proposition 7.16]{KM:suture}, we have
\begin{equation}
\label{eqn_KHI_lower_bound}
\dim_\bC\KHI(K)\ge 1.
\end{equation}
Moreover, 
$K$ is the unknot if and only if $\dim_\bC\KHI(K) = 1$. If $L$ is the unlink with $n$ components, then $\dim_\bC\KHI(L)= 2^{n-1}$ by Theorem \ref{thm_KHI_unknot} and induction on $n$. On the other hand, 
	suppose $L$ is a link with $n$ components such that $\dim_\bC\KHI(L) = 2^{n-1}$. Let $K_1,\cdots, K_n$ be the components of $L$. Then by \eqref{eqn_KHI_unknot_inequality} and \eqref{eqn_KHI_lower_bound}, we have
	$$
	\dim_\bC \KHI(K_1\cup\cdots\cup K_{m+1}) = 2\cdot \dim_\bC \KHI(K_1\cup\cdots\cup K_{m})
	$$
	for every $m<n$. Therefore by Theorem \ref{thm_KHI_unknot}, $K_{m+1}$ is an unknot in $$S^3\backslash (K_1\cup\cdots\cup K_{m})$$ 
	for every $m<n$, so $L$ is the unlink.
\end{proof}

Theorems \ref{thm_instanton_minimal_in_Y}, \ref{thm_instanton_minimal_unknot} and \ref{thm_KHI_unknot} will be proved in Section \ref{subsec_proof_thm}.

\section{Sutured manifolds and tangles}

\subsection{Balanced sutured manifolds and tangles}
This subsection reviews some terminologies on balanced sutured manifolds and tangles.

\bdefn[\cite{G:Sut-1}, \cite{Juhasz-holo-disk}]
Let $M$ be a compact oriented $3$-manifold with boundary. Suppose $\ga\subset \partial M$ is an embedded closed oriented $1$-submanifold such that 
\begin{enumerate}
	\item every component of $\partial M$ contains at least one component of $\ga$,
	\item $[\ga]=0\in H_1(\partial M;\intg).$
\end{enumerate}
Let $A(\ga)\subset \partial M$ be an open tubular neighborhood of $\ga$, and let $R(\ga)=\partial M\backslash A(\ga)$. 
Let $R_+(\ga)$ be the subset of $R(\ga)$ consisting of components whose boundaries have the same orientation as $\ga$, and let $R_-(\ga) = R(\ga) \backslash R_+(\ga)$. We say that $(M,\ga)$ is a {\it balanced sutured manifold} if it satisfies the following properties:
\begin{enumerate}
\item The $3$--manifold $M$ has no closed components.
\item $\chi(R_+(\ga))=\chi(R_-(\ga))$.	
\end{enumerate}
\edefn


\bdefn
\label{def_tangle}
Suppose $(M,\ga)$ is a balanced sutured manifold. 
\begin{enumerate}
	\item A {\it tangle} in $(M,\ga)$ is a properly embedded $1$-manifold $T\subset M$ so that $\partial T\cap \ga=\emptyset$. 
	\item A component $T_0$ of $T$ is called {\it vertical} if $\partial T_0\cap R_+(\ga)\neq\emptyset$ and $\partial T_0\cap R_-(\ga)\neq\emptyset$. 
	\item A tangle $T$ is called {\it vertical} if every component of $T$ is vertical. 
	\item A tangle $T$ is called {\it balanced} if $|T\cap R_+(\ga)| = |T\cap R_-(\ga)|$.
\end{enumerate}
\edefn

\brem
In the following, we will use the notation $T\subset (M,\ga)$ to indicate that the tangle $T$ is in $(M,\ga)$.
\erem

\brem
Note that a tangle $T\subset (M,\ga)$ is, a priori, un-oriented. When the tangle is vertical, we may orient each component of $T$ so that it goes from $R_+(\ga)$ to $R_-(\ga)$. 
\erem

\brem
The tangle $T$ in Definition \ref{def_tangle} is allowed to be empty.
\erem

\bdefn[\cite{Schar}]
Suppose $(M,\ga)$ is a balanced sutured manifold, $T\subset (M,\ga)$ is a (possibly empty) tangle, and $S\subset M$ is an oriented properly embedded surface so that $S$ intersects $T$ transversely. If $S$ is connected,  define the {\it $T$--Thurston norm} of $S$ to be
$$x_T(S)=\max\{|T\cap S|-\chi(S),0\}.$$
Here $|T\cap S|$ is the number of intersection points between $T$ and $S$. If $S$ is disconnected and
 $S_1$, ..., $S_n$ are the connected components of $S$, then define the {\it $T$--Thurston norm} of $S$ to be
$$x_T(S)=x_T(S_1)+...+x_T(S_n).$$
Suppose $N\subset \partial M$ is a subset. For a homology class $\al\in H_2(M,N;\intg)$, define the {\it $T$--Thurston norm} of $\al$ to be
\begin{align*}
	x_T(\al)= \min\{x_T(S)~|
	&~S{\rm~is~properly~embedded~in}~M,\\
	&~\partial S\subset N,~[S]=[\al]\in H_2(M,N;\intg)\}.
\end{align*}
A surface $S$ properly embedded in $M$ is said to be {\it $T$-norm-minimizing} if
$$x_T(S)=x_T([S]).$$
Here $[S]$ is the homology class in $H_2(M,\partial S;\intg)$ represented by $S$.

When $T=\emptyset$, we simply write $x_T$ as $x$.
\edefn

\bdefn[\cite{Schar}]\label{defn: taut sutured manifold}
Suppose $(M,\ga)$ is a balanced sutured manifold and $T\subset (M,\ga)$ is a (possibly empty) tangle. We say that $(M,\ga)$ is {\it $T$-taut}, if $(M,\ga,T)$ satisfies the following properties:
\begin{enumerate}
\item $M\backslash T$ is irreducible.
\item $R_{+}(\ga)$ and $R_-(\ga)$ are both incompressible in $M\backslash T$.
\item $R_{+}(\ga)$ and $R_-(\ga)$ are both $T$-norm-minimizing.
\item Every component of $T$ is either vertical or closed.
\end{enumerate}
When $T=\emptyset$, we simply say that $(M,\ga)$ is {\it taut}.
\edefn

\subsection{Sutured manifold hierarchy}
Suppose $(M,\ga)$ is a balanced sutured manifold and $T\subset (M,\ga)$ is a tangle. Suppose $S\subset M$ is an oriented properly embedded surface. Under certain mild conditions, Scharlemann \cite{Schar} generalized the concepts from \cite{G:Sut-1} and introduced a process to decompose the triple $(M,\ga,T)$ into a new triple $(M',\ga',T')$, and this process is called a \emph{sutured manifold decomposition}. Extending the work of Gabai in \cite{G:Sut-1}, Scharlemann \cite{Schar} also proved the existence of \emph{sutured manifold hierarchies} for a $T$-taut balanced sutured manifold $(M,\ga)$. In this paper, we will use a modified version of the sutured manifold hierarchy, which is presented in the current subsection. 

\bdefn[c.f. {\cite{GL-decomposition}}]\label{defn: admissible surface}
Suppose $(M,\ga)$ is a balanced sutured manifold and $T\subset (M,\ga)$ is a tangle. An oriented properly embedded surface $S\subset M$ is called {\it admissible} in $(M,\ga,T)$ if
\begin{enumerate}
\item Every component of $\partial S$ intersects $\ga$ nontrivially and transversely. 
\item For every component $T_0$ of $T$, choose an arbitrary orientation for $T_0$, then all intersections between $T_0$ and $S$ are of the same sign.
\item $S$ has no closed components.
\item The number $\frac{1}{2}|S\cap \ga|-\chi(S)$ is even. 
\end{enumerate}
\edefn

The following theorem is essentially a combination of \cite[Theorem 4.19]{Schar} and \cite[Theorem 8.2]{Juh:sut}. 
\bthm\label{thm: sutured manifold hierarchy}
Suppose $(M,\ga)$ is a balanced sutured manifold. Suppose $K$ is a knot in the interior of $M$ and view $K$ as a tangle in $(M,\ga)$. If $(M,\ga)$ is $K$--taut, then there is a finite sequence of sutured manifold decompositions
$$(M,\ga,K)\stackrel{S_1}{\leadsto}(M_1,\ga_1,K)\leadsto...\leadsto(M_{n-1},\ga_{n-1},K)\stackrel{S_n}{\leadsto}(M_n,\ga_n,T)$$
so that the following is true.
\begin{enumerate}
\item For $i=1,...,n$, the surface $S_i$ is an admissible surface in $(M_{i-1},\ga_{i-1})$. ($M_0=M$ and $\ga_0=\ga$.)
\item For $i=1,...,n-1$, we have $S_i\cap K=\emptyset$.
\item $T$ is a vertical tangle in $(M_n,\ga_n)$.
\item For $i=1,...,n-1$, $(M_i,\ga_i)$ is $K$-taut.
\item  $(M_n,\ga_n)$ is $T$-taut.
\end{enumerate}
\ethm
\bpf
The proof of the theorem is essentially a combination of the proofs of \cite[Theorem 4.19]{Schar} and \cite[Theorem 8.2]{Juh:sut}. We will only sketch the proof here and point out the adaptions to our current setup.

First, suppose there is a homology class $\al\in H_2(M,\partial M;\intg)$ so that $\partial_{*}(\al)\neq0\in H_1(\partial M;\intg)$. The argument in \cite[Section 3]{Schar} then produces a surface $S$ properly embedded in $M$ so that 
\begin{enumerate}
	\item $[S]=[\al]\in H_2(M,\partial M;\intg)$.
	\item The sutured manifold decomposition of $(M,\ga,K)$ along $S$ yields a taut sutured manifold.
\end{enumerate}
 Note that $S$ must have a non-empty boundary. We further perform the following modifications on $S$:
\begin{itemize}
\item[a)] Disregard any closed components of $S$. This makes the decomposition along $S$ resulting in a balanced sutured manifold.
\item[b)] Apply the argument in \cite[Section 5]{G:Sut-1} to make $S$ satisfying the following extra property: for every component $V$ of $R(\ga)$, the intersection of $\partial S$ with $V$ is a collection of parallel oriented non-separating simple closed curves or arcs. Note the argument in \cite[Section 5]{G:Sut-1} happens in a collar of $\partial M$ inside $M$, and the tangle $K$ is in the interior of $M$. So the argument applies to our case.
\item[c)] After step b), we can make $S$ admissible by further isotope $\partial S$ on $\partial M$ via positive stabilizations on $S$, in the sense of \cite[Definition 3.1]{L19:MinusVersion}. Note that by definition, a positive stabilization creates a pair of intersection points between $\partial S$ and $\ga$, and thus it is possible to modify $S$ so that every component of $S$ intersects $\ga$ and $\frac{1}{2}|S\cap \ga|-\chi(S)$ is even. By \cite[Lemma 3.2]{L19:MinusVersion}, we know that after positive stabilizations, the sutured manifold decomposition along $S$ still yields a $K$--taut (or $T$--taut, if $S\cap K\neq \emptyset$) balanced sutured manifold.  
\end{itemize}
Suppose $S_1$ is the resulting surface after the above modifications. If $S_1$ has a non-trivial intersection with $K$ then we are done. If $S_1\cap K=\emptyset$, we then know that $S_1$ is admissible and the sutured manifold decomposition
$$(M,\ga,K)\stackrel{S_1}{\leadsto}(M_1,\ga_1,K)$$
yields a $K$-taut balanced sutured manifold $(M_1,\ga_1)$.  We can run the above argument repeatedly. By \cite[Section 4]{Schar}, the triple $(M_1,\ga_1,K)$ has a smaller complexity than $(M,\ga,K)$, so this decomposition procedure must end after finitely many steps. There are two possibilities when the procedure ends:

{\bf Case 1}. We have a surface $S_n$ that intersects $K$ nontrivially, and $S_1$, ..., $S_{n-1}$ are all disjoint from $K$. After the decomposition along $S_n$, the knot $K$ becomes a collection of arcs $T$ inside $(M_n,\ga_n)$, and we have a sutured manifold decomposition
$$(M_{n-1},\ga_{n-1},K)\stackrel{S_n}{\leadsto}(M_n,\ga_n,T)$$
so that $(M_n,\ga_n)$ is $T$-taut. Note that $T$ has no closed components and by Condition (4) in Definition \ref{defn: taut sutured manifold}, $T$ is a vertical tangle. Hence the desired properties of the theorem are verified.

{\bf Case 2}. All decomposing surfaces $S_n$ are disjoint from $K$, and we end up with a balanced sutured manifold $(M_n,\ga_n)$ so that 
$$\partial_*: H_2(M_n,\partial M_n;\intg) \to H_1(\partial M_n; \intg)$$ 
is the zero map. In this case, the map
$$\partial_*: H_1(\partial M_n;\intg) \to H_1(M_n; \intg)$$ 
is injective.
 It then follows from a straightforward diagram chasing (see, for example, \cite[Lemma 3.5]{hatcher2007notes}) that $\partial M_n$ must be a union of $2$-spheres. Let $\hat M_n$ be the component of $M_n$ that contains $K$, and let $\hat S$ be a sphere in $\hat M_n$ that is parallel to one of the boundary components. Then $\hat S$ is a reducing sphere for $\hat M_n\backslash K$, therefore $(M_n,\ga_n)$ cannot be $K$--taut, which yields a contradiction. 
\epf

\subsection{Instanton Floer homology}\label{subsec: construction of SHI} Suppose $(M,\ga)$ is a balanced sutured manifold. Kronheimer and Mrowka constructed the sutured instanton Floer homology for the couple $(M,\ga)$ in \cite{KM:suture}. Given a balanced tangle $T\subset (M,\ga)$ (see Definition \ref{def_tangle}), an instanton Floer theory for the triple $(M,\ga,T)$ was constructed in \cite{XZ:excision}. In this subsection, we review the constructions of these instanton Floer theories.

Suppose $(M,\ga)$ is a balanced sutured manifold and $T\subset (M,\ga)$ is a balanced tangle. We construct a tuple $(Y,R,L,u)$ as follows: Pick a connected surface $F$ so that 
$$\partial F\cong-\ga.$$
Pick $n\geq2$ points $p_1$, ..., $p_n$ on $F$. Let $u$ be an embedded simple arc connecting $p_1$ and $p_2$. Glue $[-1,1]\times F$ to $M$ along $\overline{A(\ga)}$:
$$\widetilde{M}:=M\bigcup_{\overline{A(\ga)}\sim[-1,1]\times \partial F} [-1,1]\times F,$$
where $\sim$ is given by a diffeomorphism from $\overline{A(\ga)}$ to $[-1,1]\times\ga$.
The boundary of $\widetilde{M}$ consists of two components:
$$\partial \widetilde{M}= R_+\cup R_-,$$
where
$$R_{\pm}=R_{\pm}(\ga)\cup \{\pm 1\}\times F$$
are closed surfaces.
Let 
$$\widetilde{T}=T\cup [-1,1]\times\{p_1,...,p_n\}.$$

Now pick an orientation-preserving diffeomorphism
$$f: R_+\ra R_-$$
so that $f(\widetilde{T}\cap R_+)=\widetilde{T}\cap R_-$ and $f(\{1\}\times u) = \{-1\}\times u$. Gluing $R_+$ to $R_-$ via $f$ yields a closed oriented $3$-manifold $Y$. Let $R\subset Y$ be the image of $R_{\pm}$. Let $L\subset Y$ be the image of $T$, then $L$ is a link in $Y$. We will also abuse notation and let $u$ denote the image of $\{\pm1\}\times u\subset [-1,1]\times F$ in $Y$.
\bdefn
The tuple $(Y,R,L,u)$ is called a {\it closure} of $(M,\ga,T)$.
\edefn

By the construction of Kronheimer and Mrowka in \cite{KM:Kh-unknot}, there is an instanton homology group $\II(Y,L,u)$, which is a finite-dimensional $\bC$--vector space,  associated to the triple $(Y,L,u)$. The surface $R$ induces a complex linear map
$$\muu(R): \II(Y,L,u)\ra \II(Y,L,u).$$
Given a point $p\in Y\backslash L$, there is a complex linear map
$$\mu(p): \II(Y,L,u)\ra \II(Y,L,u).$$
$\muu(R)$ and $\mu(p)$ are commutative to each other. 
From now on, if $\mu$ is a linear map on $\II(Y,L,u)$, we will use 
$\Eig(\mu,\lambda)$ to denote the \emph{generalized} eigenspace of $\mu$ with eigenvalue $\lambda$ in $\II(Y,L,u)$, 

\bdefn[{\cite[Definition 7.6]{XZ:excision}}]\label{defn: SHI for (M,gamma,T)}
Suppose $(M,\ga)$ is a balanced sutured manifold, and $T\subset (M,\ga)$ is a balanced tangle. 
Let $(Y,R,L,u)$ be an arbitrary closure of $(M,\ga,T)$ such that $|L\cap R|$ is odd and at least $3$ (this can always be achieved by increasing the value of $n$ in the construction of the closure). Let $p \in  Y\backslash L$ be a point.
Then the {\it sutured instanton Floer homology} of the triple $(M,\ga,T)$, which is denoted by $\SHI(M,\ga,T)$, is defined as
$$
\SHI(M,\ga,T)= \Eig\Big(\muu(R),|L\cap R|-\chi(R)\Big)\cap \Eig\Big(\mu(p),2\Big) \subset \II(Y,L,u)
$$
\edefn

The next statement guarantees that the isomorphism class of $\SHI(M,\ga,T)$ is well-defined:

\bthm[{\cite[Proposition 7.7]{XZ:excision}}]\label{thm: isomorphism class of SHI as vector space is well-defined}
For a triple $(M,\ga,T)$, the isomorphism class of the sutured instanton Floer homology $\SHI(M,\ga,T)$ as a complex vector space is independent of the choice of the closure of $(M,\ga,T)$ and the point $p$.
\ethm

We also have the following non-vanishing theorem:
\bthm[{\cite[Theorem 7.12]{XZ:excision}}]\label{thm: nonvanishing of SHI}
Suppose $(M,\ga)$ is a balanced sutured manifold and $T\subset (M,\ga)$ is a balanced tangle so that $(M,\ga)$ is $T$--taut. Then we have
$$\SHI(M,\ga,T)\neq0.$$
\ethm

Now we establish a theorem relating sutured instanton Floer homology to sutured manifold decompositions, which is  analogous to \cite[Proposition 6.9]{KM:suture}:
\bthm
Suppose $(M,\ga)$ is a balanced sutured manifold and $T\subset (M,\ga)$ is a balanced tangle. Suppose further that $S\subset (M,\ga)$ is an admissible surface in the sense of Definition \ref{defn: admissible surface} and we have a sutured manifold decomposition:
$$(M,\ga,T)\stackrel{S}{\leadsto}(M',\ga',T').$$
Then there is a closure $(Y,R,L,u)$ of $(M,\ga,T)$ and a closure $(Y,R',L,u)$ of $(M',\ga',T')$ so that the two closures share the same $3$-manifold $Y$, the same link $L$ and the same arc $u$. Furthermore, there is a closed oriented surface $\bar{S}\subset Y$ extending $S\subset (M,\ga)$ so that $R'$ is obtained from $R$ and $\bar{S}$ by a double curve surgery in the sense of \cite[Definition 1.1]{Schar}.  Finally, let $p\in  Y\backslash L$ be a point,  then
we have
\begin{equation}\label{eq: SHI and surface decomposition, 1}
\SHI(M,\ga,T)=\Eig\Big(\muu(R),|L\cap R|-\chi(R)\Big)\cap \Eig\Big(\mu(p),2\Big)
\end{equation}
and
\begin{equation}\label{eq: SHI and surface decomposition, 2}
\begin{aligned}
	\SHI(M',\ga',T')=&\Eig\Big(\muu(R),|L\cap R|-\chi(R)\Big)\cap \Eig\Big(\mu(p),2\Big)\\
	&\cap\Eig\Big(\muu(\bar{S}),|L\cap \bar{S}|-\chi(\bar{S})\Big)
\end{aligned}	
\end{equation}
\ethm
\bpf
The proof of the theorem is almost the same as that of \cite[Proposition 6.9]{KM:suture}: the constructions of the closure $(Y,R,L,u)$ and the extension $\bar{S}$ follows verbatim from the constructions in the proof of \cite[Proposition 6.9]{KM:suture}. Equation (\ref{eq: SHI and surface decomposition, 1}) is precisely the definition of $\SHI(M,\ga,T)$ in Definition \ref{defn: SHI for (M,gamma,T)}. The verification of Equation (\ref{eq: SHI and surface decomposition, 2}) is also similar to the proof of \cite[Proposition 6.9]{KM:suture}, except that the surface in the current setup might intersect $L$ transversely at finitely many points, so we need to replace \cite[Corollary 7.2]{KM:suture} in the proof of \cite[Proposition 6.9]{KM:suture} by \cite[Proposition 6.1]{XZ:excision} for our current setup.
\epf

\brem
The original proof of \cite[Proposition 6.9]{KM:suture} only works with a surface $S\subset M$ where $\frac{1}{2}|S\cap \ga|-\chi(S)$ is even (or else it is impossible to extend $S$ to a closed surface $\bar{S}$ in the closures of $(M,\ga)$.) This is the reason why we needed to introduce the definition of admissible surfaces in Definition \ref{defn: admissible surface}. 
This subtlety was not explicitly mentioned in \cite{KM:suture}. 
\erem

\subsection{A module structure on $\SHI$}
\label{subsec_module_structure}
We introduce an additional module structure on $\SHI(M,\ga,T)$ when $T\neq \emptyset$.  

As before, suppose $(M,\ga)$ is a balanced sutured manifold and $T\subset (M,\ga)$ is a balanced tangle. Pick an arbitrary closure $(Y,R,L,u)$ of $(M,\ga)$ such that $|L\cap R|$ is odd and at least $3$ (so that the condition of Definition \ref{defn: SHI for (M,gamma,T)} is satisfied). Pick a point $p\in Y\backslash L$ and a point $q\in T$. By \cite[(2), Proposition 2.1]{XZ:excision},  the point $q$ induces another complex linear map
$$\sigma(q): \II(Y,L,u)\ra \II(Y,L,u)$$
and we have $\sigma(q)^2=0$. Moreover, $\sigma(q)$ commutes with $\mu(R)$ and $\mu(p)$. Hence $\sigma(q)$ defines a  complex linear map on 
$$\SHI(M,\ga,T)=\Eig\Big(\muu(R),|L\cap R|-\chi(R)\Big)\cap\Eig\Big(\mu(p),2\Big)$$
that squares to zero.
Therefore, we can view $\SHI(M,\ga,T)$ as a $\mathbb{C}[X]/(X^2)$--module where the action of $X$ is defined by $\sigma(q)$. The proof of {\cite[Theorem 7.12]{XZ:excision}} applies verbatim to verify the following theorem:

\bthm\label{thm: isomorphism class of SHI as module is well-defined}
For a triple $(M,\ga,T)$, the isomorphism class of the sutured instanton Floer homology $\SHI(M,\ga,T)$ as a $\mathbb{C}[X]/(X^2)$--module is independent of the choice of the closures of $(M,\ga,T)$. Moreover, if $q,q'\in T$ are in the same component of $T$, then the module structure induced by $q$ is isomorphic to the module structure induced by $q'$. \qed
\ethm

For $q\in T$, we will call the $\bC[X]/(X^2)$--module structure on $\SHI(M,\ga,T)$ defined as above the \emph{module structure  induced by $q$}. 

\begin{rem}
\label{rem_module_struct_one_component}
If $T$ has only one component, then the isomorphism class of the module structure is independent of $q$. In this case, it makes sense to refer to the isomorphism class of the module structure on $\SHI(M,\ga,T)$ without specifying the point $q$.
\end{rem}

\bcor\label{cor_free_summand_under_surface_decomposition}
Suppose $(M,\ga)$ is a balanced sutured manifold and $T\subset (M,\ga)$ is a balanced tangle. Suppose further that $S\subset (M,\ga)$ is an admissible surface in the sense of Definition \ref{defn: admissible surface} and we have a sutured manifold decomposition:
$$(M,\ga,T)\stackrel{S}{\leadsto}(M',\ga',T').$$
Suppose $q\in T$, and let $q'\in T'$ be the image of $q$ in $M'$. 
If the module structure on $\SHI(M,\ga,T)$ induced by $q$ is a free $\mathbb{C}[X]/(X^2)$--module, then so is the module structure on $\SHI(M',\ga',T')$ induced by $q'$.
\ecor
\bpf
This follows directly from Equation (\ref{eq: SHI and surface decomposition, 2}) and the fact that the action of $X$ commutes with the action of $\mu(\bar{S})$.
\epf

\subsection{$\SHI$ with local coefficients}
\label{subsec_SHI_with_local_coefficients}
We also introduce a version of sutured instanton Floer homology defined with local coefficients. Let $(M,\ga)$ be a balanced sutured manifold and let $T\subset (M,\ga)$ be a tangle. Suppose $L\subset T$ is a link in the interior of $M$.   Let $(Y,R,L',u)$ be a closure of $(M,\ga,T)$ such that $|L'\cap R|$ is odd and at least $3$. 

Let $\Ga_L$ be the local system for the triple $(Y,L',u)$ associated with $L\subset L'$ as defined in \cite[Section 2]{XZ:forest}. By definition, $\Ga_L$ is a local system over the ring $\mathcal{R}=\bC[t,t^{-1}]$.
	For $h\in\bC$, define $\Ga_L(h) = \Ga_L\otimes \mathcal{R}/(t-h)$, then $\Ga_L(h)$ is a local system over the ring $\mathcal{R}/(t-h)\cong \bC$, and hence $\II(Y,L',u; \Ga_L(h))$ is a $\bC$--linear space. Let $p\in Y\backslash L'$ be a point.
\begin{defn}
 Define
$ \SHI(M,\ga,T; \Gamma_L(h)) $
to be the simultaneous generalized eigenspace of $\muu(R)$ and $\mu(p)$ in $\II(Y,L',u; \Ga_L(h))$ with eigenvalues $|L'\cap R|-\chi(R)$ and $2$ respectively.  
\end{defn}
Since $L\cap R=\emptyset$, the proof of {\cite[Theorem 7.12]{XZ:excision}} applies verbatim to verify that: 

\bthm\label{thm: isomorphism class of SHI with local coefficien is well-defined}
For $M,\ga,T,L$ as above and $h\in \bC$, the isomorphism class of the sutured instanton Floer homology $\SHI(M,\ga,T; \Ga_L(h))$ as a $\bC$--linear space is well-defined. In other words, it is independent of the choice of the closures of $(M,\ga,T)$ and the point $p$. \qed
\ethm

\begin{lem}
	\label{lem:_limit_as_h_to_1}
	$$
	\limsup_{h\to 1} \, \dim_\bC \SHI(M,\ga,T;\Ga_L(h)) \le \dim_\bC \SHI (M,\ga,T).
	$$
\end{lem}

\begin{proof}
	This is an immediate consequence of \cite[Lemma 4.1 (1)]{XZ:forest} and the observation that $\Ga_L(1)$ is the trivial local system over $\bC$, thus
	$$
	\SHI (M,\ga,T;\Ga_L(1)) = \SHI (M,\ga,T).
	\phantom\qedhere\makeatletter\displaymath@qed
	$$
\end{proof}

\begin{lem}
\label{lem:_isomorphism_class_of_SHI_with local coefficient is invariant under homotopy}
Suppose $T_0$ is a tangle in $(M,\ga)$, and $L, L'\subset M$ are two links in the interior of $M$ that are homotopic to each other in $M$ and are both disjoint from $T_0$.  Let $T = T_0\cup L$, $T'=T_0\cup L'$, and suppose $h \notin \{\pm1\}$. Then 
$$
\dim_\bC (M,\ga,T;\Ga_L(h)) = \dim_\bC (M,\ga,T';\Ga_{L'}(h))
$$
\end{lem}

\begin{proof}
	This is an immediate consequence of \cite[Proposition 3.2]{XZ:forest}.
\end{proof}

\begin{lem}
\label{lem_add_unknot_local_coef}
	Suppose $T$ is a tangle in $(M,\ga)$ and $U$ is an unknot in $M\backslash T$. Then for every $h\in \bC$, we have
	$$
	\dim_\bC\SHI(M,\ga,T\cup U,\Ga_U(h)) = 2\cdot  \dim_\bC\SHI(M,\ga,T).
	$$
\end{lem}

\begin{proof}
Let $B$ be a $3$--ball in $M\backslash(T\cup U)$, let $\delta$ be a simple closed curve on $\partial B$. Let $U'\subset S^3$ be an unknot in $S^3$, let $B'$ be a  $3$--ball in $S^3\backslash U'$, let $\delta'$ be a simple closed curve on $\partial B'$.
	By the connected sum formula for $\SHI$ (c.f. \cite[Proposition 4.15]{L18-contact}), we have
\begin{multline*}
	\dim_{\mathbb{C}}\SHI(M\backslash B, \ga\cup \delta, T\cup U;\Ga_U(h))  \\
	 = 2\cdot\dim_\bC\SHI(M, \ga, T\cup U; \Ga_U(h)) \cdot \dim_\bC\SHI(S^3\backslash B', \delta'),
\end{multline*}
and
\begin{multline*}
	\dim_{\mathbb{C}}\SHI(M\backslash B, \ga\cup \delta, T\cup U;\Ga_U(h)) \\
	= 2\cdot\dim_\bC\SHI(M, \ga, T) \cdot \dim_\bC\SHI(S^3\backslash B', \delta',U'; \Ga_{U'}(h)).
\end{multline*}
Therefore
\begin{align*}
	& \dim_\bC\SHI(M, \ga, T\cup U; \Ga_U(h)) \cdot \dim_\bC\SHI(S^3\backslash B', \delta')\\
=~ & \dim_\bC\SHI(M, \ga, T) \cdot \dim_\bC\SHI(S^3\backslash B', \delta',U'; \Ga_{U'}(h)).
\end{align*}
Since $(S^3\backslash B', \delta')$ is a product sutured manifold, we have 
$$\dim_\bC\SHI(S^3\backslash B', \delta') =1.$$
Let
$$c(h) = \dim_\bC\SHI(S^3\backslash B', \delta',U'; \Ga_{U'}(h)),$$ 
then 
$c(h)$ is a function 
depending on $h$, and we have 
$$
\dim_\bC\SHI(M,\ga,T\cup U,\Ga_U(h)) =  c(h)\cdot  \dim_\bC\SHI(M,\ga,T).
$$
for all $M,\ga,T$. Recall that $T$ is allowed to be empty. Now we invoke the computations of annular instanton Floer homology $\AHI$ with local coefficients from \cite[Example 3.4]{XZ:forest}. Since $\AHI$ can be regarded as a special case of sutured instanton Floer homology, by \cite[Example 3.4]{XZ:forest}, we must have $c(h)=2$ for all $h$. 
\end{proof}

\begin{rem}
The $\bC[X]/(X^2)$--module structure introduced in Section \ref{subsec_module_structure} cannot be straightforwardly extended to 
	 $\SHI$ with local coefficients. The reason is that \cite[(2)]{XZ:forest} may no longer hold for instanton Floer homology with local coefficients, so we may not have $\sigma(q)^2 = 0$ on $\SHI(M,\ga,T; \Ga_L(h))$.
\end{rem}

\section{Properties of Floer minimal knots}

\subsection{Irreducibility and module structure}
This subsection proves the following theorem. The result is inspired by the work of Wang \cite{wang2021link}. 
\bthm\label{thm: free implies reducible, case of one taut sutured manifolds}
Suppose $(M,\ga)$ is a balanced sutured manifold with
\begin{equation}
\label{eqn_SHI_non_zero_irreducible}
	\SHI(M,\gamma)\neq 0
\end{equation} 
and $K$ is a knot in the interior of $M$. View $K$ as a balanced tangle in $(M,\ga)$. If $M\backslash K$ is irreducible, then 
$\SHI(M,\ga,K)$ is not free as a $\mathbb{C}[X]/(X^2)$--module.
\ethm

\bpf
Since $M\backslash K$ is irreducible, by \eqref{eqn_SHI_non_zero_irreducible} and the adjunction inequality for instanton Floer homology \cite[Theorem 7.21]{KM:suture}, we conclude that $(M,\ga)$ is $K$--taut.

Let
$$(M,\ga,K)\stackrel{S_1}{\leadsto}(M_1,\ga_1,K)\leadsto...\leadsto(M_{n-1},\ga_{n-1},K)\stackrel{S_n}{\leadsto}(M_n,\ga_n,T)$$
be the sutured hierarchy given by Theorem \ref{thm: sutured manifold hierarchy}. In particular, $T$ is a vertical tangle and $(M_{n},\ga_{n})$ is $T$-taut.

Assume $\SHI(M,\ga,K)$ is a free $\mathbb{C}[X]/(X^2)$-module. 
Applying Corollary \ref{cor_free_summand_under_surface_decomposition} repeatedly and applying Theorem \ref{thm: nonvanishing of SHI}, we conclude that for every point $q\in T$, the module structure on $\SHI(M_n,\ga_n,T)$ induced by $q$ is a non-trivial free $\mathbb{C}[X]/(X^2)$--module. 

Pick any closure $(Y,R,L,u)$ of $(M_n,\ga_n,T)$ such that $|R\cap L|$ is odd and at least $3$. 
Let $N(R)\cong [-1,1]\times R$ be a closed tubular neighborhood of $R$ in $Y$.
Since $R\cap T\neq\emptyset$, we can choose $q$ so that it is contained in $N(R)$. Isotope the arc $u$ so that $u$ intersects $R$ transversely at one point $p_0$.  To simplify  notation, we will identify $N(R)$ with $[-1,1]\times R$.
 After shrinking $N(R)$ and taking isotopies for $L$ and $u$, we may assume that 
\begin{enumerate}
	\item $L\cap N(R)$ is given by $[-1,1]\times\{p_1,...,p_m\}\subset [-1,1]\times R$
	\item $u\cap N(R)$ is given by $[-1,1]\times\{p_0\}\subset [-1,1]\times R$.
\end{enumerate} 
 
 Now we can cut $Y$ open along the two surfaces $\{-1\}\times R$ and $\{1\}\times R$, and re-glue via the diffeomorphism from $\{-1\}\times R$ to $\{1\}\times R$ defined by the identity map on $R$. This yields a disjoint union $Y_1\sqcup Y_2$, where 
 \begin{enumerate}
 	\item  $Y_1$ is given by $(Y\backslash [-1,1]\times R)/\sim$, where $\sim$ is the identification of $\{-1\}\times R$ with $\{1\}\times R$,
 	\item $Y_2$ is given by $([-1,1]\times R)\slash\sim$, where $\sim$ is the same identification as above.
 \end{enumerate}
Then $Y_1$ is diffeomorphic to $Y$, and $Y_2$ is diffeomorphic to $S^1\times R$. 
 
Let $R_1$ be the image of $\{\pm 1\}\times R$ in $Y_1$, let $L_1$, $u_1$ be the images of $L$ and $u$ in $Y_1$ respectively.
Clearly $(Y_1,R_1,L_1,u_1)$ is diffeomorphic to $(Y,R,L,u)$. 
In $Y_2$, let $R_2=\{0\}\times R$, let $L_2=S^1\times\{p_1,...,p_n\}$, and let $u_2=S^1\times \{p_0\}$. As in \cite[Theorem 6.4]{XZ:excision}, there is a cobordism $W$ from $Y_1\sqcup Y_2$ to $Y$ that induces an isomorphism (as $\bC$--linear spaces)
$$e:\II(Y_1,L_1,u_1|R_1)\otimes \II(Y_2,L_2,u_2|R_2)\xra{\cong} \II(Y,L,u|R),$$
where the notation $\II(Y_i,L_i,u_i|R_i)$ is defined in \cite[Definition 6.3]{XZ:excision}.
Recall that we assumed the base point $q$ is in the neighborhood $[-1,1]\times R$. Let $q_2\in Y_2$ be the image of $q$. Then $\sigma(q_2)$ acts on $I(Y_2,L_2,u_2|R_2)$. By Section \ref{subsec_module_structure}, we have $\sigma(q_2)^2=0$. On the other hand, by \cite[Proposition 6.5]{XZ:excision},
$$\II(Y_2,L_2,u_2|R_2)\cong\mathbb{C}.$$
This implies that $\sigma(q_2)=0$ on $\II(Y_2,L_2,u_2|R_2)$, so $\sigma(q_2)$ defines the zero action on the tensor product
$$\II(Y_1,L_1,u_1|R_1)\otimes \II(Y_2,L_2,u_2|R_2).$$
However, the cobordism map $e$ intertwines the $\sigma(q_2)$--action on
$$\II(Y_1,L_1,u_1|R_1)\otimes \II(Y_2,L_2,u_2|R_2)$$
with the $\sigma(q)$--action on $\SHI(M_n,\ga_n,T)$. Thus we conclude that the action of $X$ (which equals $\sigma(q)$) on $\SHI(M_n,\ga_n,T)$ must also be zero. 
As a consequence, the module structure on $\SHI(M_n,\ga_n,T)$ induced by $q$ is not a free $\mathbb{C}[X]/(X^2)$--module,
which yields a contradiction.
\epf

\subsection{Floer minimality and module structure}
Recall that for a sutured manifold $(M,\ga)$, the group $\pi_1'(M)$ is defined by Equation \eqref{eqn_def_pi'} as a quotient group of $\pi_1(M)$. The main result of this subsection is the following theorem.

\begin{thm} 
\label{thm_floer_simple_implies_SHI_free}
Let $(M,\ga)$ be a balanced sutured manifold. Suppose $K$ is a knot in the interior of $M$ that represents the zero element in $\pi_1'(M)$. If $K$ is instanton Floer minimal (see Definition \ref{def_instanton_minimal}), then $\SHI(M,\ga,K)$ is a free module over $\bC[X]/(X^2)$ (see Section \ref{subsec_module_structure}).
\end{thm}

We first establish the following rank inequality for instanton homology.

\blem
\label{lem_rank_inequality_K_zero_in_pi'}
Suppose $(M,\ga)$ is a balanced sutured manifold, and $K$ is a knot in the interior of $M$ such that $K$ represents the zero element in $\pi_1'(M)$. Then we have
\begin{equation}
\label{eqn_rank_inequality_K_zero_in_pi'}
\dim_{\mathbb{C}}\SHI(M,\ga,K)\ge  2\cdot\dim_{\mathbb{C}}\SHI(M,\ga).
\end{equation}
\elem

\begin{proof}
Suppose $\ga=\ga_1\cup\cdots\cup\ga_m$ has $m$ components. For each $i$, let $\ga_i',\ga_i''$ be a pair of oppositely-oriented simple closed curves on $\partial M$ parallel to $\ga_i$. Moreover, let $\ga_i''$ have the same orientation as $\ga$, let $\ga'$ have the opposite orientation as $\ga$, and let $\ga'$ be placed between $\ga$ and $\ga''$. Let $\ga' := \cup_i(\ga_i\cup\ga_i'\cup\ga_i'')$. Let $U$ be an unknot in $M$. The proof of \cite[Theorem 3.1]{KM:Alexander} implies that
\begin{align}
	\dim_\bC \SHI(M,\ga') &= 2^m\cdot \dim_\bC\SHI(M,\ga) 
	\label{eqn_add_ga'_1}
	\\
	\dim_\bC \SHI(M,\ga',K) &= 2^m \cdot \dim_\bC\SHI(M,\ga,K) 
	\label{eqn_add_ga'_2} 
	\\
	\dim_\bC \SHI(M,\ga',U) &= 2^m\cdot \dim_\bC\SHI(M,\ga,U).
	\label{eqn_add_ga'_3}
\end{align}
Decompose $A(\ga')$ as $A(\ga) = \cup_i(A(\ga_i)\cup A(\ga_i')\cup A(\ga_i''))$. For each $i$, let $D_i$ be a disk, let $\varphi_i$ be an orientation-reversing diffeomorphism from $[-1,1]\times \partial D_i$ to $A(\ga_i')$, and let 
$$\hat M = M\cup_{\{\varphi_i\}} (\cup_i [-1,1]\times D_i) $$
Let $\hat \ga$ be the image of $\cup_i(\ga_i\cup\ga_i'')$ on $\partial \hat M$. Then every component of $\partial \hat M$ contains at least one component of $\hat \ga$. Moreover, it is straightforward to verify that $(\hat M, \hat \ga)$ is a balanced sutured manifold. Let $T_i=[-1,1]\times\{0\}\subset [-1,1]\times D_i$ be a tangle, we can view $T=T_1\cup...\cup T_m$ as a vertical tangle inside $(\hat{M},\hat{\gamma})$. Let $(\hat{M}_{T},\hat{\ga}_T)$ be the balanced sutured manifold obtained from $(\hat M,\hat \ga)$ by removing a neighborhood of $T$ and adding a meridian of every component of $T$ to the suture. It is straightforward to check that $(\hat{M}_{T},\hat{\ga}_T)$ is simply diffeomorphic to $(M,\ga')$. Hence by \cite[Lemma 7.10]{XZ:excision}, we have
$$
\SHI(M,\ga')\cong \SHI(\hat M, \hat \ga, T).
$$
In fact, the proof of \cite[Lemma 7.10]{XZ:excision} also applies verbatim to give
\begin{align*}
	\SHI(M,\ga',K) &\cong \SHI(\hat M,\hat \ga,T \cup K), \\
	\SHI(M,\ga',U) &\cong \SHI(\hat M,\hat \ga,T \cup U).
\end{align*}

By Lemma \ref{lem_add_unknot_local_coef}, we have
$$
\dim_{\mathbb{C}}\SHI(\hat M,\hat \ga, T\cup U;\Ga_U(h)) = 2\cdot\dim_{\mathbb{C}}\SHI(\hat M,\hat \ga, T) 
$$
for every $h\in \bC$. 
Since $K$ represents the zero element in $\pi_1'(M)$, it is homotopic to $U$ in $\hat M$.
By Lemma \ref{lem:_isomorphism_class_of_SHI_with local coefficient is invariant under homotopy}, when $h\notin \{\pm1\}$, we have 
$$
\dim_{\mathbb{C}}\SHI(\hat M,\hat \ga, K\cup T;\Ga_K(h)) 
 = \dim_{\mathbb{C}}\SHI(\hat M,\hat \ga, K\cup U; \Ga_U(h)). 
$$
Hence
$$
\dim_{\mathbb{C}}\SHI(\hat M,\hat \ga, K\cup T;\Ga_K(h)) = 2\cdot\dim_{\mathbb{C}}\SHI(\hat M,\hat \ga, T)
$$
for all $h\notin \{\pm1\}$.
Therefore, by Lemma \ref{lem:_limit_as_h_to_1},
\begin{align*}
\dim_{\mathbb{C}}\SHI(M,\ga',K) 
= ~ 
&\dim_{\mathbb{C}} \SHI(\hat M,\hat \ga,T \cup K) \\
\ge ~
&2\cdot\dim_{\mathbb{C}}\SHI(\hat M,\hat \ga, T) \\
= ~
&2\cdot\dim_{\mathbb{C}}\SHI(M,\ga').
\end{align*}
Thus by \eqref{eqn_add_ga'_1} and \eqref{eqn_add_ga'_2}, we have
$$
\dim_{\mathbb{C}}\SHI(M,\ga,K) 
\ge  2\cdot \dim_{\mathbb{C}}\SHI(M,\ga),
$$
and  the lemma is proved.
\end{proof}

Now we can prove Theorem \ref{thm_floer_simple_implies_SHI_free}.

\bpf[Proof  of Theorem \ref{thm_floer_simple_implies_SHI_free}]
	We define a new instanton Floer homology group $\SHI(M,\ga,K^{\natural})$ as follows. Regard $K$ as a tangle inside $(M,\ga)$. Pick a closure $(Y,R,L,u)$ of $(M,\ga,K)$. Recall that from the construction of closures in Section \ref{subsec: construction of SHI}, $L$ has the form
$$L=K\cup L_1,$$
where $L_1$ is the image of the arcs $[-1,1]\times \{p_1,...,p_n\} \subset [-1,1]\times F$ in the closure. Now define $L'=K\cup \mu\cup L_1$ to be a link in $Y$, where $\mu$ is a meridian of $K$, and define $u'=u\cup v$, where $v$ is an arc connecting $K$ with $\mu$.  Define
$$\SHI(M,\ga,K^{\natural}):=\II(Y,L',u'|R).$$

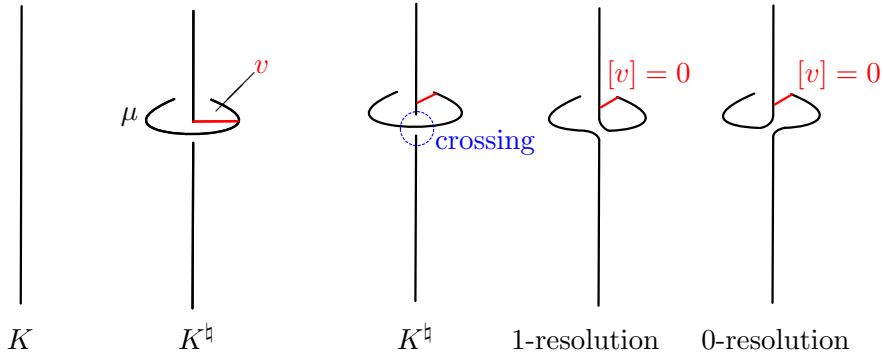
\begin{figure}[h]
\centering
\begin{overpic}[width=\textwidth]{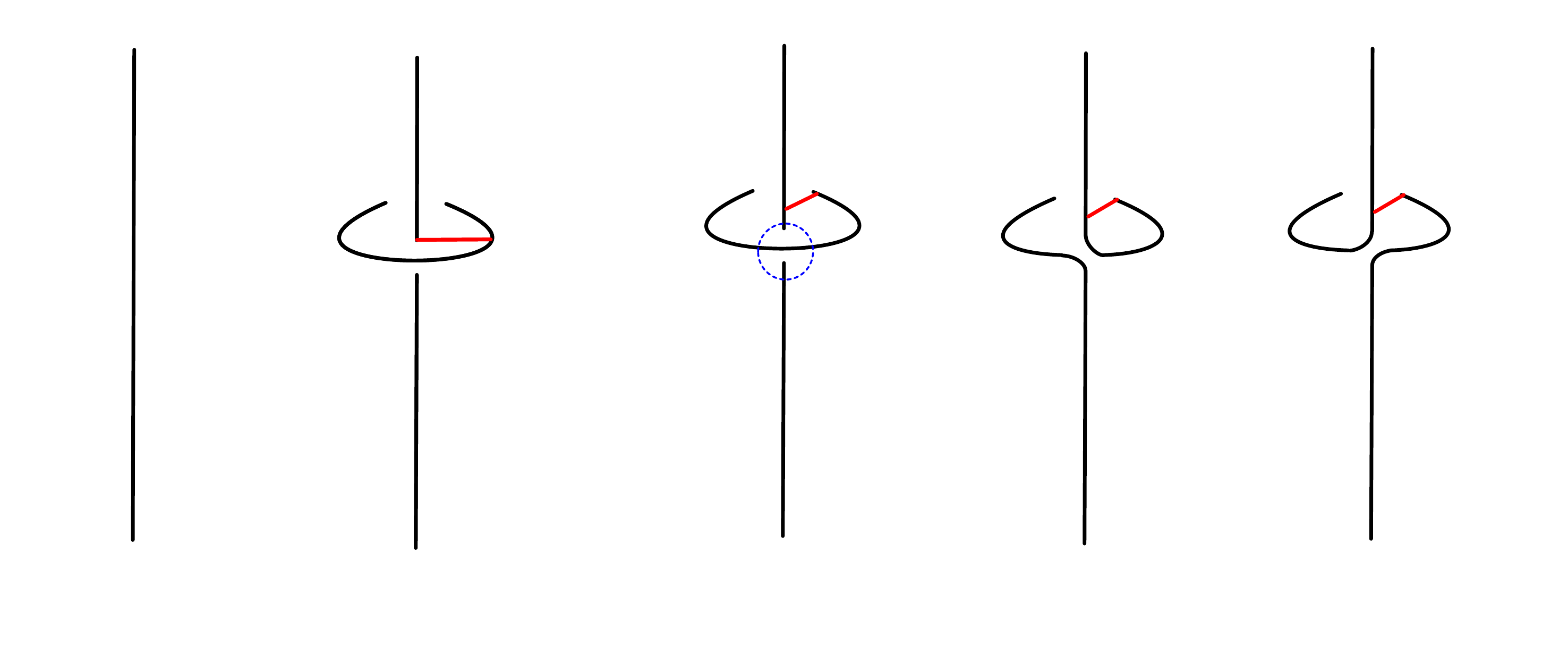}
	\put(7,2){$K$}
	\put(25,2){$K^{\natural}$}
	\put(19,26){$\mu$}
	\put(29,27){\line(1,1){4}}
	\put(33,31){\color{red}$v$}
	\put(48,2){$K^{\natural}$}
	\put(52,23){\color{blue}crossing}
	\put(60,2){$1$-resolution}
	\put(70,30){\color{red}$[v]=0$}
	\put(90,30){\color{red}$[v]=0$}
	\put(80,2){$0$-resolution}
\end{overpic}
\caption{The knot $K$, $K^{\natural}$, and the unoriented skein triangle.}\label{fig: K and K-natural}
\end{figure}

\brem
If we replace $(M,\ga)$ by a closed $3$-manifold $Y$ and disregard $R$, $L_1$, and $u$, then the above construction is the same as the definition of $\II^{\natural}(Y,K)$ introduced by Kronheimer and Mrowka in \cite{KM:Kh-unknot}.
\erem

Now apply the un-oriented skein triangle from \cite[Section 6 and 7]{KM:Kh-unknot} to a crossing between $K$ and $\mu$, as shown in Figure \ref{fig: K and K-natural}. The result is the following exact triangle
\begin{equation*}
	\xymatrix{
	\SHI(M,\ga,K)\ar[rr]^{f}&&\SHI(M,\ga,K)\ar[dl]\\
	&\SHI(M,\ga,K^{\natural})\ar[lu]&
	}
\end{equation*}
The map $f$ has been computed explicitly in \cite[Proposition 3.4]{Xie-earring}, and we have
$$f=\pm\sigma(q):\SHI(M,\ga,K)\ra \SHI(M,\ga,K)$$
for a base point $q\in K$. Since $\sigma(q)^2=0$, we have
\begin{equation}\label{eq: dimension inequality between K and K-natural}
	\dim_{\mathbb{C}}\SHI(M,\ga,K)\leq\dim_{\mathbb{C}}\SHI(M,\ga,K^{\natural}),
\end{equation}
and the equality holds if and only if
\begin{equation}\label{eq: sigma q is half rank}
{\rm rank}(\sigma(q))=\frac{1}{2}\dim_{\mathbb{C}}\SHI(M,\ga,K).
\end{equation}

Let $\ga_{K}$ be the union of $\ga$ with a pair of oppositely oriented meridians of $K$ on $\partial N(K)$. 
The proof of \cite[Proposition 1.4]{KM:Kh-unknot} (also \cite[Lemma 7.10]{XZ:excision}) applies verbatim to give the following isomorphism
\begin{equation}
\label{eqn_K_natural_iso_suture}
	\SHI(M,\ga,K^{\natural})\cong \SHI(M\backslash N(K),\ga_K).
\end{equation}
Recall that $K$ is assumed to be instanton Floer minimal, namely
\begin{equation}\label{eq: dimension relation between K and emptyset}
\dim_{\mathbb{C}}\SHI(M\backslash N(K),\ga_K) = 2\cdot \dim_{\mathbb{C}}\SHI(M,\ga).
\end{equation}
By Lemma \ref{lem_rank_inequality_K_zero_in_pi'}, we have
\begin{equation}
\label{eqn_rank_inequality_remove_contractible_knot}
	\dim_{\mathbb{C}}\SHI(M,\ga,K)\geq 2\cdot\dim_{\mathbb{C}}\SHI(M,\ga).
\end{equation}
By \eqref{eq: dimension inequality between K and K-natural}, \eqref{eqn_K_natural_iso_suture}, \eqref{eq: dimension relation between K and emptyset}, \eqref{eqn_rank_inequality_remove_contractible_knot}, we conclude that
$$\dim_{\mathbb{C}}\SHI(M,\ga,K)=\dim_{\mathbb{C}}\SHI(M,\ga,K^{\natural}),$$
therefore the equality holds for (\ref{eq: dimension inequality between K and K-natural}), and hence (\ref{eq: sigma q is half rank}) holds. Recall that  the action of $X$ equals $\sigma(q)$ in the $\mathbb{C}[X]/(X^2)$--module structure of $\SHI(M,\ga,K)$. Hence by (\ref{eq: sigma q is half rank}), $\SHI(M,\ga,K)$ is a free $\mathbb{C}[X]/(X^2)$-module.
\epf

\subsection{Proofs of the main results} 
\label{subsec_proof_thm}

This subsection proves Theorems \ref{thm_instanton_minimal_in_Y}, \ref{thm_instanton_minimal_unknot}, and \ref{thm_KHI_unknot}.

\begin{proof}[Proof of Theorem \ref{thm_instanton_minimal_in_Y}]
By Theorem \ref{thm_floer_simple_implies_SHI_free}, $\SHI(M,\ga,K)$ is a non-trivial free module over $\bC[X]/(X^2)$. Hence by Theorem \ref{thm: free implies reducible, case of one taut sutured manifolds}, $M\backslash K$ is reducible.

Therefore, $M$ decomposes as $M=M_1\# M_2$, where $K\subset M_2$ after isotopy. If $M_2$ is a closed manifold, then the desired result holds. If $M_2$ is a manifold with boundary,
since $\SHI(M, \ga)\neq 0$, by \cite[Theorem 7.21]{KM:suture}, $R_+(\ga)$ and $R_-(\ga)$ are both norm-minimizing, therefore $M_2$ is also a balanced sutured manifold.
By the connected sum formula for $\SHI$, we have 
$$
\dim_\bC \SHI(M_2,\ga,K) = 2\cdot \dim_\bC\SHI(M_2,\ga).
$$
Hence the desired result follows from induction on the number of factors in the prime decomposition of $M$.
\end{proof}

\begin{proof}[Proof of Theorem \ref{thm_instanton_minimal_unknot}]
By the connected sum formula for $\SHI$ and the non-vanishing theorem for taut sutured manifolds \cite[Theorem 7.12]{KM:suture}, we have $\SHI(M,\ga)\neq 0$.

By Theorem \ref{thm_instanton_minimal_in_Y}, there exists a $3$--ball $B^3\subset M$ such that $K\subset B^3$. Hence we can view $K$ as a knot in the northern hemisphere of $S^3$.
Define $S^3(K) = S^3\backslash N(K)$, and let $\ga_\mu\subset \partial S^3(K)$ be a pair of oppositely oriented meridians of $K$. Let $\ga_K\subset \partial(M\backslash N(K))$ be the union of $\ga$ and a pair of oppositely oriented meridians on $\partial N(K)$. 
The connected sum formula for $\SHI$ implies
$$\dim_{\mathbb{C}}\SHI(M\backslash N(K),\ga_K)=2\cdot \dim_{\mathbb{C}}\SHI(M,\ga)\cdot \dim_{\mathbb{C}}\SHI(S^3(K),\ga_{\mu}).$$
Therefore $K$ is instanton Floer minimal if and only if
$$\II^\natural(S^3,K) \cong \SHI(S^3(K),\ga_{\mu})\cong \mathbb{C}$$
By the unknot detection result for $\II^\natural(S^3,K)$ \cite[Proposition 7.16]{KM:suture}, the above equation holds if and only if $K$ is the unknot.
\end{proof}

\begin{proof}[Proof of Theorem \ref{thm_KHI_unknot}]
	Let $M=S^3\backslash N(L)$, then the prime decomposition of $M$ does not contain any non-trivial closed component.
	Let $\ga\subset \partial M$ be the suture consisting of a pair of oppositely oriented meridians for each component of $L$, then $\pi_1'(M) = 0$. By \cite[Proposition 7.16]{KM:suture}, we have $\SHI(M,\ga)\neq 0$, therefore $M$ satisfies the conditions of Theorem \ref{thm_instanton_minimal_unknot}, and hence \eqref{eqn_KHI_unknot} holds if and only if $K$ is the unknot in $M$.
\end{proof}

\bibliographystyle{amsalpha}
\bibliography{references}

\begin{bibdiv}
\begin{biblist}

\bib{G:Sut-1}{article}{
      author={Gabai, David},
       title={Foliations and the topology of {$3$}-manifolds},
        date={1983},
        ISSN={0022-040X},
     journal={J. Differential Geom.},
      volume={18},
      number={3},
       pages={445\ndash 503},
         url={http://projecteuclid.org/euclid.jdg/1214437784},
      review={\MR{723813 (86a:57009)}},
}

\bib{GL-decomposition}{article}{
      author={Ghosh, Sudipta},
      author={Li, Zhenkun},
       title={Decomposing sutured monopole and instanton floer homologies},
        date={2019},
     journal={arXiv:1910.10842, v2},
}

\bib{GL:simple_knots}{article}{
      author={Greene, Joshua~Evan},
      author={Luecke, John},
       title={Fibered simple knots},
        date={2021},
     journal={arXiv:2106.08485},
}

\bib{hatcher2007notes}{misc}{
      author={Hatcher, Allen},
       title={Notes on basic 3-manifold topology},
        date={2007},
}

\bib{hedden2011floer}{article}{
      author={Hedden, Matthew},
       title={On floer homology and the berge conjecture on knots admitting
  lens space surgeries},
        date={2011},
     journal={Transactions of the American Mathematical Society},
      volume={363},
      number={2},
       pages={949\ndash 968},
}

\bib{Juhasz-holo-disk}{article}{
      author={Juh\'{a}sz, Andr\'{a}s},
       title={Holomorphic discs and sutured manifolds},
        date={2006},
        ISSN={1472-2747},
     journal={Algebr. Geom. Topol.},
      volume={6},
       pages={1429\ndash 1457},
         url={https://doi.org/10.2140/agt.2006.6.1429},
      review={\MR{2253454}},
}

\bib{Juh:sut}{article}{
      author={Juh{\'a}sz, Andr{\'a}s},
       title={Floer homology and surface decompositions},
        date={2008},
        ISSN={1465-3060},
     journal={Geom. Topol.},
      volume={12},
      number={1},
       pages={299\ndash 350},
         url={http://dx.doi.org/10.2140/gt.2008.12.299},
      review={\MR{2390347}},
}

\bib{Kim}{article}{
      author={Kim, Juhyun},
       title={Links of second smallest knot {F}loer homology},
        date={2020},
     journal={arXiv:2011.11810},
}

\bib{KM:Alexander}{article}{
      author={Kronheimer, P.~B.},
      author={Mrowka, T.~S.},
       title={Instanton {F}loer homology and the {A}lexander polynomial},
        date={2010},
        ISSN={1472-2747},
     journal={Algebr. Geom. Topol.},
      volume={10},
      number={3},
       pages={1715\ndash 1738},
         url={https://doi.org/10.2140/agt.2010.10.1715},
      review={\MR{2683750}},
}

\bib{KM:suture}{article}{
      author={Kronheimer, P.~B.},
      author={Mrowka, T.~S.},
       title={Knots, sutures, and excision},
        date={2010},
     journal={Journal of Differential Geometry},
      volume={84},
      number={2},
       pages={301\ndash 364},
}

\bib{KM:Kh-unknot}{article}{
      author={Kronheimer, P.~B.},
      author={Mrowka, T.~S.},
       title={Khovanov homology is an unknot-detector},
        date={2011},
        ISSN={0073-8301},
     journal={Publ. Math. Inst. Hautes \'Etudes Sci.},
      number={113},
       pages={97\ndash 208},
         url={http://dx.doi.org/10.1007/s10240-010-0030-y},
      review={\MR{2805599}},
}

\bib{L18-contact}{article}{
      author={Li, Zhenkun},
       title={Contact structures, excisions, and sutured monopole {F}loer
  homology},
        date={2018},
     journal={arXiv:1811.11634},
}

\bib{L19:MinusVersion}{article}{
      author={Li, Zhenkun},
       title={Knot homologies in monopole and instanton theoreis via sutures},
        date={2019},
     journal={arXiv:1901.06679, v6},
}

\bib{LY-HeegaardDiagram}{article}{
      author={Li, Zhenkun},
      author={Ye, Fan},
       title={{Instanton Floer homology, sutures, and {H}eegaard diagrams}},
        date={2020},
     journal={arXiv: 2010.07836},
}

\bib{Ni:rank_unlink}{article}{
      author={Ni, Yi},
       title={Homological actions on sutured {F}loer homology},
        date={2014},
        ISSN={1073-2780},
     journal={Math. Res. Lett.},
      volume={21},
      number={5},
       pages={1177\ndash 1197},
         url={https://doi.org/10.4310/MRL.2014.v21.n5.a12},
      review={\MR{3294567}},
}

\bib{Ni-Wu:floer_simple}{article}{
      author={Ni, Yi},
      author={Wu, Zhongtao},
       title={Heegaard {F}loer correction terms and rational genus bounds},
        date={2014},
        ISSN={0001-8708},
     journal={Adv. Math.},
      volume={267},
       pages={360\ndash 380},
         url={https://doi.org/10.1016/j.aim.2014.09.006},
      review={\MR{3269182}},
}

\bib{OS:HFK}{article}{
      author={Ozsv\'{a}th, Peter},
      author={Szab\'{o}, Zolt\'{a}n},
       title={Holomorphic disks and knot invariants},
        date={2004},
        ISSN={0001-8708},
     journal={Adv. Math.},
      volume={186},
      number={1},
       pages={58\ndash 116},
         url={https://doi.org/10.1016/j.aim.2003.05.001},
      review={\MR{2065507}},
}

\bib{rasmussen2007lens}{article}{
      author={Rasmussen, Jacob},
       title={Lens space surgeries and l-space homology spheres},
        date={2007},
     journal={arXiv:0710.2531},
}

\bib{rasmussen2017floer}{article}{
      author={Rasmussen, Jacob},
      author={Rasmussen, Sarah~Dean},
       title={Floer simple manifolds and {L}-space intervals},
        date={2017},
     journal={Advances in Mathematics},
      volume={322},
       pages={738\ndash 805},
}

\bib{Schar}{article}{
      author={Scharlemann, Martin},
       title={Sutured manifolds and generalized {T}hurston norms},
        date={1989},
        ISSN={0022-040X},
     journal={J. Differential Geom.},
      volume={29},
      number={3},
       pages={557\ndash 614},
         url={http://projecteuclid.org/euclid.jdg/1214443063},
      review={\MR{992331}},
}

\bib{wang2021link}{article}{
      author={Wang, Joshua},
       title={Link floer homology also detects split links},
        date={2021},
     journal={Bulletin of the London Mathematical Society},
}

\bib{Xie-earring}{article}{
      author={Xie, Yi},
       title={Earrings, sutures and pointed links},
     journal={to appear in International Mathematics Research Notices},
         url={https://doi.org/10.1093/imrn/rnaa101},
}

\bib{XZ:forest}{article}{
      author={Xie, Yi},
      author={Zhang, Boyu},
       title={Classification of links with {K}hovanov homology of minimal
  rank},
        date={2019},
     journal={arXiv:1909.10032},
}

\bib{XZ:excision}{article}{
      author={Xie, Yi},
      author={Zhang, Boyu},
       title={Instanton {F}loer homology for sutured manifolds with tangles},
        date={2019},
     journal={arXiv:1907.00547},
}

\end{biblist}
\end{bibdiv}

\end{document}